\documentclass[12pt,oneside,a4paper]{article}
\usepackage{fancyhdr, ifpdf}
\usepackage{syntonly}
\usepackage{geometry}
\usepackage{makeidx}
\makeindex
\usepackage{setspace}

\onehalfspace
\author{Safari Mukeru \footnote{Department of Decision Sciences, School of Economic Sciences, College of Economic and Management Sciences, University of South Africa, Muckleneuk Campus, P. O. Box 392, Pretoria, 0003. South Africa. {\it email: Mukers@unisa.ac.za}}} 

\title{Hausdorff dimension and capacities of compact sets} 
\date{}

\begin{document}
\vspace{6in}
\maketitle

\pagenumbering{arabic}

\abstract{We give a complete proof of the expression of capacities of a measure in terms of its Fourier transform}

\paragraph{Keywords:} Fourier transform, Capacity, measure, Hausdorff dimension, Fourier dimension. 

\newtheorem{theorem}{Theorem}
\newtheorem{lemma}[theorem]{Lemma}
\newtheorem{proposition}[theorem]{Proposition}
\newtheorem{corollary}[theorem]{Corollary}
\newtheorem{definition}[theorem]{Definition}
\newtheorem{example}[theorem]{Example}
\newtheorem{remark}[theorem]{Remark}
\newenvironment{proof}[1][Proof]{\begin{trivlist}
\item[\hskip \labelsep {\bfseries #1}]}{\end{trivlist}}

\newcommand{\qed}{\nobreak \ifvmode \relax \else
      \ifdim\lastskip<1.5em \hskip-\lastskip
      \hskip1.5em plus0em minus0.5em \fi \nobreak
      \vrule height0.30em width0.4em depth0.25em\fi}

\section{Introduction}

Given a Radon finite measure $\mu$ supported by a compact subset $E$ of $\mathbf{R}^d$ and a real number $\alpha$ such that $0 < \alpha < d$ the energy integral of $\mu$ with respect to the kernel $k(x)= \vert x \vert ^{-\alpha}$ (or simply the $\alpha-$energy of $\mu$) is given by 
\begin{eqnarray} \label{energyy}
I_\alpha(\mu) = \int\int \frac{d\mu(x) d\mu(y)}{\vert x-y\vert^\alpha}.\end{eqnarray}
The measure $\mu$ is said to have finite energy with respect to $k$ if $I_\alpha(\mu) < \infty$.  The set $E$ has positive capacity with respect to $k$ and write $\mbox{Cap}_\alpha(E) > 0$ if $E$ carries a non-zero Radon measure of finite energy with respect to $k$. If there is no such measure $E$ is said to have zero capacity with respect to this kernel and we write $\mbox{Cap}_\alpha(E) = 0$. 
By the celebrated Frostman theorem \cite{Frostman}, the Hausdorff dimension of $E$ is equal to 
$\sup\{\alpha: \mbox{Cap}_\alpha(E) > 0\} = \inf\{\beta: \mbox{Cap}_\beta (E) = 0 \}$. To show that a given compact subset $E$ has Hausdorff dimension $\geq \alpha$, it is sufficient to construct a non-zero finite Radon measure $\mu$ such that $I_\alpha(\mu) < \infty$. That is why capacities are very important in fractal geometry. 
 (An interested reader can find more details on capacities and fractal geometry in the books by Kahane \cite{Kahane} and by Mattila \cite{Mattila}).  
It is known that, the $\alpha-$energy of $\mu$ is also given by 
 \begin{eqnarray} \label{capac_four_formula}
 I_\alpha(\mu) = \frac{1}{(2\pi)^n}\int \hat k(u) \vert \hat \mu(u) \vert ^2 du,
\end{eqnarray}
where $\hat \mu$ is the Fourier transform of $\mu$ and $\hat k$ the Fourier transform of the kernel $k$ in a sense to be precise later on. This formula is a corner stone in the Fourier analysis of fractal properties of sets because in some sense it is easier to apply than computing the double integral in (\ref{energyy}). It is also the only one formula known to the author that relates directly the Fourier transform of a measure on a fractal set to the Fourier transform of the kernel $k$. It is well known that if $E$ is a compact subset of $[0, 1]$ of Hausdorff dimension $\alpha$, and if  $\mu$ is a probability measure whose 
support is contained in $E$, then the function $\vert u \vert ^\beta \hat \mu(u)$ is unbounded for any $\beta > \alpha/2$. This is result is based on relation (\ref{capac_four_formula}).

However in the literature, there is no complete proof of this result known to the author. The most cited proof is given in the book by Carelson \cite[pp 22-23]{Carleson} and contains just some few lines. The proof by Mattila in his book \cite[pp162-163]{Mattila} , even though more detailed than the one by Carleson, contains also many gaps and exercises that are not obvious to fill. Due to the importance of this formula we have decided in this paper to improve and complete Mattila's proof by giving more precision on the involved constants and filling all the gaps. The construction will be used to study (in collaboration with Fouch\'e) the Hausdorff dimension of images of compact subsets by algorithmically Brownian motion. 

We start by giving, in Section 2, some basics on Fourier transform of functions and tempered distribution in order to calculate the Fourier transform of the kernel $k$. This is borrowed in the book by Strichartz \cite{Strichartz}. Next, in Section 3, we discuss Fourier transform of measures and convolution products. The proof itself is given in Section 4.

\section{Fourier transform  of integrable functions}
Given a function function $f \in L^{1}(\mathbf{R^n})$, its Fourier transform is the function defined by
\begin{eqnarray} \label{four_trans_funct}
\hat f(u) = \int e^{iux} f(x)dx, \,\, u \in \mathbf{R}^n.
\end{eqnarray}
The function $\hat f$ is continuous.

If $\hat f$ is also summable, that is, $\hat f \in L^{1}(\mathbf{R^n})$, then we have the following Fourier inversion formula (\cite{Rudin}, p 185): 
\begin{eqnarray} \label{inverse_four_trans}
f(x) = \frac{1}{(2\pi)^n}\int e^{-iux} \hat f(u)du, \, \mbox{ almost everywhere in } \mathbf{R}^n.
\end{eqnarray}
If $f_1, \, f_2 \in L^{1}(\mathbf{R^n})$ and their Fourier transforms $\hat f_1, \hat f_2$ belong to $L^{2}(\mathbf{R^n})$, then we have the formula (\cite{Rudin}, p 187):
\begin{eqnarray} \label{parseval_formula}
\frac{1}{(2\pi)^n}\int \hat f_1(u) \overline{\hat f_2(u)}du =  \int f_1(x) \overline{f_2(x)} dx. 
\end{eqnarray}
In particular, 
$$\frac{1}{(2\pi)^n}\int \vert \hat f_1(u)\vert^2 du =  \int \vert f_1(x) \vert ^2 dx.$$
The Fourier transform of the Gaussian function $$f(x) = e^{-s\vert x \vert^2}, \,\, s > 0$$ is given by (see for example \cite{Strichartz}, pp 38-41)
\begin{eqnarray} \label{Four_trand_Gauss}
\hat f(u) = (\pi/s)^{n/2} e^{-\vert u \vert^2/4s}
\end{eqnarray}
 and will be used in the sequel.

\section{Fourier transform of tempered distributions}
We will need Fourier transforms of functions that are not necessarily in $L^{1}(\mathbf{R^n})$ but are locally integrable. The usual way to define their Fourier transforms is to consider them as tempered distributions. \\
Let us introduce the following common notations:
For any $\alpha = (\alpha_1, \ldots, \alpha_n) \in \mathbf{N}^n$ and $x = (x_1, \ldots, x_n) \in \mathbf{R}^n$, 
\begin{eqnarray*}
\vert \alpha\vert &=& \alpha_1 + \ldots+ \alpha_n \\
x^\alpha &=& x_1^{\alpha_1}\ldots x_n^{\alpha_n}\\
\partial^\alpha& = &\frac{\partial^{\vert \alpha\vert}}{\partial^{\alpha_1}\ldots \partial^{\alpha_n}}.
\end{eqnarray*}
Consider an open subset $W$ of $\mathbf{R}^n$ and the linear space $C^\infty_0(W)$ of $C^\infty$-functions defined on $W$ having compact support. This space can be endowed by the structure of locally convex topological space as follows (\cite{Al-Gwaiz}, pp 24-25):
\begin{enumerate}
\item Write $$C^\infty_0(W) = \cup_{K \in \mathcal{K}} C^\infty_K(W)$$ where $\mathcal{K}$ is the class of all compact subsets of $W$, and $C^\infty_K(W)$ the subset of $C^\infty_0(W)$ whose elements have support in $K$. 
\item Endow $C^\infty_K(W)$ with the topology defined by the family of norms
  $$p_i(\phi) = \sup\{ \vert\partial ^\alpha \phi(x)\vert: x \in K, \vert \alpha \vert \leq i\}, \, i \in \mathbf{N}$$
that is, consider the neighborhood system of zero to be the family of balls
$$B_i(r) = \{\phi \in C^\infty_K(W): p_i(\phi) < r\}, \, r > 0 \,\mbox{ and } i \in \mathbf{N}.$$    
\item Endow $C^\infty_0(W)$ with the inductive limit topology of the topologies on the $C^\infty_K(W)$'s, that is, the neighborhood system of zero is the class of subsets $U \in C^\infty_0(W)$ such that $U \cap C^\infty_K(W)$ is a neighborhood of zero in $C^\infty_K(W)$ and that $U$ is convex and balanced in the sense that, for any $f \in U$ and $\lambda \in \mathbf{C}$ such that $\vert \lambda \vert \leq 1$, $\lambda f \in U$ holds.
\end{enumerate}
The space $C^\infty_0(W)$ endowed with this topology is denoted $\mathcal{D}(W)$. 
\begin{definition}
A distribution on $W$ is a continuous linear functional on $\mathcal{D}(W)$.
\end{definition}
The set of distributions on $W$ is denoted $\mathcal{D}'(W)$. For any locally integrable function $f$ on $W$, the linear functional 
$$\langle f, \phi\rangle =  \int_W f(x) \phi(x)dx,\,\, \phi \in \mathcal{D}(R^n),$$ is a distribution on $W$. Another important example of distribution is the classical Dirac distribution $\delta_x, \, x \in \mathbf{R^n}$ defined by
$$ \langle\delta_x, \phi \rangle = \phi(x), \, \phi \in \mathcal{D}(R^n).$$ We will simply denote $\delta_0$ as $\delta$. Any Borel measure $\mu$ with compact support defined on $\mathbf{R^n}$ induces a distribution on $\mathbf{R}$ by
$$\langle \mu, \phi \rangle = \int \phi(x) d\mu(x), \,\phi \in \mathcal{D}(R^n) $$
  \begin{definition}
 A sequence $(T_n)$ in $\mathcal{D}'(W)$ converges (weakly) to $T$ if for every $\phi \in \mathcal{D}(W)$, the sequence $\langle T_n,\phi \rangle $ converges to $\langle T,\phi\rangle$ in $\mathbf{C}$.  
 \end{definition}
 Usually, it is useful to consider the Dirac distribution as the limit of a sequence of integrable functions. The following proposition gives such a sequence (\cite{Al-Gwaiz}, pp 48-49):

\begin{proposition}
For any nonnegative integrable function $f$ on $\mathbf{R}^n$ such that \\ $\int f(x) dx = 1$, the family$(f_\epsilon), \epsilon > 0$ defined by
$$f_\epsilon (x) = \frac{1}{\epsilon^n}f\left(\frac{x}{\epsilon}\right)$$ converges to $\delta$ in $\mathcal{D}'(R^n)$ as $\epsilon \to 0$.
\end{proposition}
An example is given by 
$$f_\epsilon (x) = \frac{1}{(\pi \epsilon)^{-n/2}} e^{-\vert x \vert ^2/\epsilon}.$$
We know turn to a specific class of distributions called, tempered distributions, on which it is possible to extend Fourier transform. \\
\begin{definition}
A function $\phi \in C^\infty(\mathbf{R})$ is said to be rapidly decreasing if
$$\sup_{x \in \mathbf{R}^n} \vert x^\alpha \partial ^\beta \phi(x) \vert < \infty $$ for all multi-indices $\alpha$ and $\beta$. This is equivalent to 
$$\sup_{\vert \beta \vert \leq m} \sup_{x \in \mathbf{R}^n} \left(1 + \vert x \vert ^2\right)^m \vert \partial ^\beta \phi(x) \vert < \infty,$$  for all integers $m \geq 1$.
\end{definition}
The space of rapidly decreasing functions on $\mathbf{R}^n$ is denoted $\mathcal{S}(\mathbf{R}^n)$. It is a linear topological space where the topology is defined by the family of semi-norms $p_{\alpha,\beta}$, $\alpha, \, \beta \in \mathbf{N}_0^n$ such that
$$p_{\alpha,\beta}(\phi) = \sup_{x \in \mathbf{R}^n} \vert x^\alpha \partial ^\beta \phi(x) \vert,\,\, \,\phi \in \mathcal{S}(\mathbf{R}^n).$$
The topological space $\mathcal{D}(\mathbf{R}^n)$ is a dense subspace of $\mathcal{S}(\mathbf{R}^n)$. 
\begin{definition}
A tempered distribution is a continuous linear function on $\mathcal{S}(\mathbf{R}^n)$. 
\end{definition} 
Their set is denoted $\mathcal{S}'(\mathbf{R}^n)$. Clearly, $\mathcal{S}'(\mathbf{R}^n) \subset \mathcal{D}'(\mathbf{R}^n)$.  
 \begin{example} \label{exemple_temepp}
Many properties of fractal geometry are based on the following function defined on $\mathbf{R}^n - \{0\}$ by:
$$k(x) = \frac{1}{\vert x \vert ^\alpha}, \mbox{ for some } 0 \leq \alpha < n.$$ 
It is very useful to see that $k$ defines a tempered distribution ($k$ can be extended at at $0$ by taking for example $k(0) = \infty$). 
 \end{example}
In deed, for any $\phi \in \mathcal{S}(\mathbf{R}^n)$, consider as usual,
$$\langle f, \phi \rangle = \int k(x) \phi(x) dx.$$ It is sufficient to show that $\vert \langle f, \phi \rangle \vert $ is finite. We have, for some real $A >0$, that,
\begin{eqnarray*}
\vert \langle f, \phi \rangle \vert &\leq &\int \vert k(x) \phi(x)\vert dx \\
                                    & = & \int_{\vert x \vert \geq A}  k(x) \vert \phi(x)\vert dx + \int_{\vert x \vert < A}  k(x) \vert \phi(x)\vert dx 
\end{eqnarray*}
For $\vert x \vert \geq A$,  $k(x) \leq 1/A^\alpha$ and for $\vert x \vert < A$, there exists $M > 0$ such that 
$\vert \phi(x) \vert \leq M$ (since $\phi$ is bounded). Therefore,
$$\vert \langle f, \phi \rangle \vert \leq \frac{1}{A^\alpha} \int_{\vert x \vert \geq A} \vert \phi(x)\vert dx + M \int_{\vert x \vert < A}  k(x) dx.$$
It remains to show that $\int_{\vert x \vert \geq A} \vert \phi(x)\vert dx < \infty$. Since $\phi \in \mathcal{S}(\mathbf{R}^n)$, we have that 
$$\sup_{x \in \mathbf{R}^n} (1 + \vert x \vert ^2)^n \vert \phi(x) \vert < \infty$$ and then
$$\vert \phi(x) \vert = (1 + \vert x \vert ^2)^{-n} (1 + \vert x \vert ^2)^n \vert \phi(x) \vert \leq H (1 + \vert x \vert ^2)^{-n}, \mbox{ for some } H > 0.$$
Therefore
$$\int \vert \phi(x) \vert dx \leq H \int \frac{dx}{(1 + \vert x \vert ^2)^n} < \infty.$$ 

We have that $\mathcal{S}(\mathbf{R}^n) \subset L^{1}(\mathbf{R}^n)$ and it is true in general that 
$\mathcal{S}(\mathbf{R}^n) \subset L^{p}(\mathbf{R}^n)$ for any $1 \leq p \leq \infty$. Also, $L^{p}(\mathcal{R}^n) \subset \mathcal{S}'(\mathbf{R}^n)$, (\cite{Al-Gwaiz}, p 122).\\

Then the Fourier transform of any $\phi \in \mathcal{S}(\mathbf{R}^n)$ exists. One of the most important properties of the space 
$\mathcal{S}(\mathbf{R}^n)$ is that the Fourier transform defines a continuous linear operator in $\mathcal{S}(\mathbf{R}^n)$ and the Fourier inversion formula holds, that is: 
$$\phi \in \mathcal{S}(\mathbf{R}^n) \Longrightarrow \hat \phi \in \mathcal{S}(\mathbf{R}^n) \mbox{ and if } \phi_n \to \phi \mbox{ then } \hat \phi_n \to \hat \phi$$
and for any $ x\in \mathbf{R}^N$, $$\phi(x) = \frac{1}{(2\pi)^n}\int e^{-iux} \hat \phi(u)du, \,\, \phi \in \mathcal{S}(\mathbf{R}^n).$$
(See for example, Theorems 4.3 and 4.4 in \cite{Al-Gwaiz}, p 124--125).\\
We are now ready to extend the Fourier transform operator on tempered distributions.
\begin{definition} \label{Fou_trans_distr}
For any $T \in \mathcal{S}'(\mathbf{R}^n)$, the Fourier transform $\hat T$ of $T$ is defined by
$$\langle \hat T, \phi \rangle = \langle T, \hat \phi \rangle.$$
\end{definition}
If $\phi \in  \mathcal{S}(\mathbf{R}^n)$, then $\phi$ defines a tempered distribution $T_\phi$ whose Fourier transform is denoted $\hat T_\phi$. The Fourier transform $\hat \phi$ also defines a tempered distribution since it is an element of $\mathcal{S}(\mathbf{R}^n)$. Let us denote it by $T_{\hat \phi}$. Then we have that $\hat T_\phi = T_{\hat \phi}$. 

Let us now find the Fourier transform of the tempered distribution defined by the function $k(x) = 1/x^\alpha.$
\begin{proposition}
For any $0 \leq \alpha < n$, the Fourier transform of the tempered distribution defined by the function 
$k(x) = \frac{1}{\vert x \vert^\alpha}$ is the tempered distribution defined by the function
\begin{eqnarray} \label{Fou_kernel}
\hat k(u) = \frac{\pi^{n/2} 2^{\alpha + n} \Gamma(\alpha/2 + n/2)}{\Gamma(\alpha/2)} \vert u \vert ^{\alpha - n},\,\, u \in \mathbf{R}^n,
\end{eqnarray}
where $$\Gamma (z) = \int_0^\infty t^{z-1} e^{-t}dt$$ is the gamma function.
 \end{proposition}
\begin{proof} The following proof is adapted from the book by Strichartz (\cite{Strichartz}, pp 50--51).
We start by considering the integral
$$I = \int_0^\infty s^{\frac{\alpha}{2} - 1} e^{-s \vert x \vert^2}ds $$
With the variable change $h = s \vert x \vert ^2$, we find, that, 
$$I = \frac{1}{\vert x \vert ^\alpha} \int_0^\infty s^{\frac{\alpha}{2} - 1} e^{-s}ds = \frac{\Gamma(\alpha/2)}{\vert x \vert ^\alpha}.$$
Therefore we have the identity
 $$\frac{1}{\vert x \vert ^\alpha} = \frac{1}{\Gamma(\alpha/2)} \int_0^\infty s^{\frac{\alpha}{2} - 1} e^{-s \vert x \vert^2}ds, \,\, x \ne 0.$$
 For any $\phi \in \mathcal{S}(\mathbf{R}^n)$, we have that,
 \begin{eqnarray*}
\langle \hat k, \phi \rangle &= & \langle k, \hat\phi \rangle\\
                             & = &\int_{\mathbf{R}^n} k(x) \hat \phi(x) dx \\
                             & = & \frac{1}{\Gamma(\alpha/2)} \int_{\mathbf{R}^n} \left( \int_0^\infty s^{\frac{\alpha}{2} - 1} e^{-s \vert x \vert^2}ds \right) \hat \phi(x) dx \\
                             & = & \frac{1}{\Gamma(\alpha/2)} \int_{\mathbf{R}^n} \left( \int_0^\infty s^{\frac{\alpha}{2} - 1} e^{-s \vert x \vert^2}ds \right) \int_{\mathbf{R}^n} e^{ixz}\phi(z)dz dx\\
                             & = & \frac{1}{\Gamma(\alpha/2)} \int_{\mathbf{R}^n} \int_0^\infty \left(\int_{\mathbf{R}^n} e^{ixz}e^{-s \vert x \vert^2}dx\right)s^{\frac{\alpha}{2} - 1} \phi(z) ds dz \mbox{  by Fubini's theorem}\\
                             & = &  \frac{1}{\Gamma(\alpha/2)} \int_{\mathbf{R}^n} \int_0^\infty (\pi/s)^{n/2} e^{-\vert z \vert^2/4s}  s^{\frac{\alpha}{2} - 1} \phi(z) ds dz\,\, \mbox{ (from relation } (\ref{Four_trand_Gauss}))\\        & = & \frac{\pi^{n/2}}{\Gamma(\alpha/2)} \int_{\mathbf{R}^n} \left( \int_0^\infty s^{\frac{\alpha}{2} - \frac{n}{2} -1} e^{-\vert z \vert^2/4s} \phi(z) ds \right) dz \\
& = &   \frac{\pi^{n/2} 2^{n - \alpha}}{\Gamma(\alpha/2)} \int_{\mathbf{R}^n} \left( \int_0^\infty \vert z \vert^{\alpha -n} e^{-h} h^{-\frac{\alpha}{2} + \frac{n}{2} - 1} dh \right) \phi(z) dz \mbox{ (by taking }  s = \vert z \vert ^2/4h) \\ 
& = & \frac{\pi^{n/2} 2^{n - \alpha} \Gamma((n-\alpha)/2)}{\Gamma(\alpha/2)} \int_{\mathbf{R}^n}  \vert z \vert ^{\alpha-n} \phi(z) dz\\
& = & c(\alpha, n)\langle g, \phi \rangle        
\end{eqnarray*}
where 
\begin{eqnarray*}
c(\alpha, n) &=& \frac{\pi^{n/2} 2^{n - \alpha} \Gamma((n-\alpha)/2)}{\Gamma(\alpha/2)} \,\, \mbox{ and }\\
g(z) &=& \vert z \vert ^{\alpha -n}.
\end{eqnarray*}
It follows that $\hat k$ is the tempered distribution defined by the function $c(\alpha, n)g$. \qed
\end{proof}

\section{Fourier transform of measures}
We will need the notion of Fourier transform of Radon measures of compact support. 
\begin{definition}
The Fourier transform of a Radon measure $\mu$ on $\mathbf{R}^n$ of compact support is the function defined by
$$\hat \mu(u) = \int e^{iux}d\mu(x), \, \, u \in \mathbf{R}^n.$$
\end{definition}
Since $\mu(\mathbf{R}^n) < \infty$, then $\hat \mu$ is a bounded uniformly continuous function. In the sequel, it will be useful to approximate measures  by convolution products. We consider the following definitions: 
\begin{definition}
Let $f$ and $g$ be real functions on $\mathbf{R}^n$ and $\mu$ be a Radon measure of compact support on $\mathbf{R}^n$. The convolutions $f*g$ of $f$ and $g$, and $f*\mu$ of $f$ and $\mu$ are defined by
\begin{eqnarray*}
f*g(x) &=& \int f(x-y)g(y)dy \\
f*\mu(x) &=& \int f(x-y)d\mu(y)
\end{eqnarray*}
provided the integral exists.
\end{definition} \label{conv_prodd}
Clearly, we have the following:
\begin{eqnarray}
 \widehat{f*g} = \hat f \hat g\, \textrm{ and } \widehat{f*\mu} = \hat f \hat \mu
\end{eqnarray}
provided the involved integrals exist.
\begin{definition} \label{definition_approximate_identity}
An approximate identity $(\phi_\epsilon)_{\epsilon > 0}$ is a family of nonnegative continuous functions on $\mathbf{R}^n$ such that the support of each $\phi_\epsilon$ is contained in the ball $B(\epsilon)$ of centre $0$ and radius $\epsilon$ and $\int \psi_\epsilon dx = 1.$
\end{definition}
Such families are usually constructed by taking a continuous function $f: \mathbf{R}^n \to [0, \infty)$ such that its support is contained in $B(1)$ and $\int f(x)dx =1$ and consider 
$$ \phi(x) = \epsilon^{-n} f(x/\epsilon), \,\, \epsilon > 0.$$
We have the following proposition (\cite{Mattila}, p 20):
\begin{proposition} \label{limit_conv}
Let $(\phi_\epsilon)$ be an approximate identity.
\begin{enumerate}
\item[$(1)$] If $\mu$ is a compactly supported Radon measure defined on $\mathbf{R}^n$, then the family of functions
$\phi_{\epsilon} * \mu$ converges weakly to $\mu$ as $\epsilon$ tends to $0$, in the sense that
$$\lim_{\epsilon \to 0} \int f(x)(\phi_\epsilon * \mu)(x)dx = \int f(x) d\mu(x)$$ for any uniformly continuous bounded function $f$ defined on $\mathbf{R}^n$.
\item[$(2)$] If $g$ is a bounded function defined on an open subset $W$ of $\mathbf{R}^n$ and continuous at $x$, then 
$$\lim_{\epsilon \to 0} g*\phi_\epsilon(x) = g(x).$$
\end{enumerate}
\end{proposition}
\begin{proof}
 (1) Because $\int \psi_\epsilon dx = 1,$ we have that 
 \begin{eqnarray*}
&&\int f(x)(\phi_\epsilon * \mu)(x)dx - \int f(x)d\mu(x) \\
 &=& \int f(x)\left( \int \phi_\epsilon(x - y) d\mu(y)\right) dx - \int f(y)d\mu(y) \int \phi_\epsilon(x) dx \\
 & = & \int \left( \int f(x) \phi_\epsilon(x - y) dx \right) d\mu(y) - \int\left( \int f(y)\phi_\epsilon(x) dx \right) d\mu(y)\, \mbox{ (by Fubini's theorem)}\\
 & = & \int \left( \int f(h+y) \phi_\epsilon(h) dh \right) d\mu(y) - \int\left( \int f(y)\phi_\epsilon(x) dx \right) d\mu(y)\, \,\mbox{ (by taking } h = x - y)\\
& = &  \int \left( \int (f(x+y) - f(y)) \phi_\epsilon(x) dx \right) d\mu(y).
\end{eqnarray*}
Since $f$ is uniformly continuous and bounded, then for any $\gamma>0$, there exists $\delta > 0$ such that for any $h, y$ with $\vert h-y\vert < \delta$ we have
$\vert f(h) - f(y) \vert < \gamma$. Then by taking  $\epsilon >0$ sufficiently small such that $\epsilon < \delta$, we have that for any $x \in B(\epsilon)$, and any $y \in \mathbf{R}^n$, $\vert (x+ y)- y \vert < \delta$ and hence $\vert f(x+y) - f(y) \vert < \gamma$. Using the fact $B(\epsilon)$ contains the support of  $\phi_\epsilon$, one finds that
\begin{eqnarray*}
\left\vert \int f(x)(\phi_\epsilon * \mu)(x)dx - \int f(x)d\mu(x)\right\vert& =& \left\vert\int \left( \int_{B(\epsilon)} (f(x+y) - f(y)) \phi_\epsilon(x) dx \right) d\mu(y) \right\vert\\
& \leq& \int \int \gamma \phi_\epsilon(x) dx d\mu(y) \\
& = & \gamma \mu(\mathbf{R}^n).
\end{eqnarray*}
Since $\mu(\mathbf{R}^n) < \infty$, we conclude that 
$$\lim_{\epsilon \to 0} \int f(x)(\phi_\epsilon * \mu)(x)dx - \int f(x)d\mu(x) = 0.$$
(2) As previously, we write 
$$g * \phi_\epsilon (x) - g(x) = \int_{B(\epsilon)} (g(x-t) - g(x)) \phi_\epsilon (t) dt.$$
Since $g$ is continuous at $x$, for any $\gamma > 0$, there exists $\delta >0$ such that 
$\vert g(y) - g(x) \vert < \gamma$ holds for any $y$ such that $\vert y -x \vert < \delta$. By taking $y = x - t$ and $\epsilon < \delta$, we find $\vert g(x-t) - g(x) \vert < \gamma$, for any $\vert t \vert < \epsilon$. It follows that $$\vert g * \phi_\epsilon (x) - g(x) \vert \leq \gamma \int \phi_\epsilon (t) dt = \gamma$$
and hence $\lim_{\epsilon \to 0} g*\phi_\epsilon(x) = g(x).$ \qed
\end{proof}

\section{Fourier transform and Capacities}
Consider a compact subset $E$ of $\mathbf{R}^n$. The energy integrals
$$I_\alpha(\mu) = \int\int k(x-y)d\mu(x) d\mu(y)= \int\int \frac{d\mu(x) d\mu(y)}{\vert x-y\vert^\alpha}$$ are very useful in the calculation of the Hausdorff dimension of $E$. \\
The following theorem which relates $I_\alpha(\mu)$ to $\hat\mu$ is an important result of fractal geometry (\cite{Carleson}, pp 22-23), (\cite{Mattila}, pp 162-163).  
\begin{theorem} \label{Impooth}
For any Radon measure $\mu$ on $\mathbf{R}^n$ with compact support, and any \\
 $0 \leq \alpha < n$, 
\begin{eqnarray} \label{capac_four_formula_2}
 I_\alpha(\mu) = \frac{1}{(2\pi)^n}\int \hat k(u) \vert \hat \mu(u) \vert ^2 du,
\end{eqnarray}
 where $\hat k$ is given by relation $(\ref{Fou_kernel})$.
\end{theorem}

\begin{proof} 
We approximate $k$ by the convolution product of $k$ by an approximate identity. Consider an approximate identity $(\phi_\epsilon)$ defined by 
$\phi_\epsilon(x) = \epsilon ^{-n} f(x/\epsilon)$ where $f$ is a $C^\infty$-function whose support is the ball $B(1/2)$ of radius $1/2$ and centre $0$ and such that $\int f(x)dx = 1$. It is clear that $\psi_\epsilon = \phi_\epsilon * \phi_\epsilon$ is also an approximate identity and we have that $\psi(x) = \epsilon^{-n} f*f(x/\epsilon)$. We will also assume that $f(x) = f(-x)$ and that the Fourier transform $\hat f$ is a nonnegative function. \\
For any $x \ne 0$, the function $k$ is continuous at $x$ and from Proposition \ref{limit_conv}, \\$\lim_{\epsilon \to 0} k*\psi_\epsilon (x) = k(x)$. This is also true for $x = 0$ if we take $k(0) = \infty$.  Then, by Fatou's lemma,
\begin{eqnarray*}
I_\alpha(\mu) = \int\int k(x-y) d\mu(x) d\mu(y) &\leq & \liminf_{\epsilon \to 0}\int \int k*\psi_\epsilon (x-y) d\mu(x) d\mu(y). \end{eqnarray*}
(1) First we want to show that 
\begin{eqnarray} \label{firsteqq}
\liminf_{\epsilon \to 0}\int \int k*\psi_\epsilon (x-y) d\mu(x) d\mu(y) \leq \frac{1}{(2\pi)^n}\int \hat k(u) \vert \hat \mu(u) \vert ^2 du.\end{eqnarray}
In order to make use of Fubini's theorem to compute these integrals, we need to show that $$\int \int k*\psi_\epsilon (x-y) d\mu(x) d\mu(y) < \infty.$$  We have that,
 for any $z \ne 0$ in $\mathbf{R}^n$, if $\frac{\vert z \vert}{2} > \epsilon$ then for any $u \in B(\epsilon)$,
\begin{eqnarray} \label{eqn_tempor}
 \frac{1}{\vert z - u\vert^\alpha} \leq \frac{2^{-\alpha}}{\vert z \vert ^{\alpha}}.
 \end{eqnarray}
Indeed, the function  $ u \to \vert z - u\vert^\alpha$ attains its minimum at 
$u = \epsilon z /\vert z \vert$ and hence 
$$\sup_{u\in B(\epsilon)}\frac{\vert z \vert ^{\alpha}}{\vert z - u \vert^\alpha} = \frac{1}{(1 - \epsilon/\vert z \vert)^\alpha} \leq 2^{-\alpha}.$$
Now for $x-y \ne 0$, and $\vert x- y\vert/2 > \epsilon$, we find
$$k*\psi_\epsilon (x-y) = \int_{B(\epsilon)} k(x-y-u)\psi_\epsilon(u) du \leq 2^{-\alpha}k(x-y)\int_{B(\epsilon)}\psi_\epsilon(u) du  = 2^{-\alpha} k(x-y).$$ 
For $\vert x-y \vert/2 \leq \epsilon$, we have that
\begin{eqnarray*}
k*\psi_\epsilon (x-y) &=& \int_{B(\epsilon)} k(x-y-u)\psi_\epsilon(u) du \\
&\leq & H \int_{B(\epsilon)} k(x-y-u)du  \mbox{ (where }H = \sup \psi_\epsilon)\\
&\leq & H \int_{B(3\epsilon)} \frac{dt}{\vert t\vert^\alpha} \mbox{ (since }\vert x-y\vert \leq 2 \epsilon)\,\,(t = x- y -u)\\
& = & C(\epsilon).
\end{eqnarray*}
Therefore $k*\psi_\epsilon (x-y)$ is bounded by a constant depending only on $\epsilon$. Now it follows that 
$$\int \int k*\psi_\epsilon (x-y) d\mu(x) d\mu(y) \leq 2^{-\alpha} \int \int k(x-y) d\mu(x) d\mu(y) + C(\epsilon) A  < \infty$$ 
where $A = (\mu(\mathbf{R}^n))^2$.\\

(2) Now we have that,
\begin{eqnarray*}
k*\psi_\epsilon (x-y) & = & (k * \phi_\epsilon) * \phi_\epsilon(x-y)\\
                     & = &  \int k*\phi_\epsilon (x - y -h) \phi_\epsilon (h)dh \\
                     & = & \int k*\phi_\epsilon (h- x + y)) \phi_\epsilon (h)dh \mbox{ by symmetry of } k \mbox{ and } \phi_\epsilon \\
                     &= & \int k*\phi_\epsilon (z + y) \phi_\epsilon (z+x) dz \mbox{ by taking }z = h-x \\
                     & = & \int \left(\int k(t) \phi_\epsilon (z+y - t) dt \right) \phi_\epsilon (z+ x) dz.
\end{eqnarray*}
Using Fubini's theorem, we find
\begin{eqnarray*}
&& \int\int k*\psi_\epsilon (x-y) d\mu(x) d\mu(y) = \\
&&\int \int k(t) \left(\int \phi_\epsilon (z+y - t)d\mu(y)\right) \left( \int \phi_\epsilon (z+ x) d\mu(x)\right) dt dz.
 \end{eqnarray*}
The first inner integral is
\begin{eqnarray*}
\int \phi_\epsilon (z+y - t)d\mu(y) & = & \int \phi_\epsilon (t - z-y)d\mu(y) \mbox{ (by symmetry  of } \phi_\epsilon) \\
                                    & = & \phi_\epsilon * \mu(t - z)
\end{eqnarray*}
and the second is 
\begin{eqnarray*}
\int \phi_\epsilon (z+ x) d\mu(x) = \int \phi_\epsilon (z- x) d\tilde \mu(x) = \phi_\epsilon * \tilde \mu(z)
\end{eqnarray*}
where $\tilde \mu$ is the Radon measure defined by $\int g(x) d \tilde \mu(x) = \int g(-x) d\mu(x)$ for any continuous function $g$ on $\mathbf{R}^n$ with compact support.
Therefore, 
\begin{eqnarray*}
\int\int k*\psi_\epsilon (x-y) d\mu(x) d\mu(y) &=& \int k(t)\int \phi_\epsilon * \mu(t - z) \phi_\epsilon * \tilde \mu(z) dz dt \\
& = & \int k(t) (\phi_\epsilon * \mu)*(\phi_\epsilon * \tilde \mu)(t) dt.
\end{eqnarray*}
(3) We can now pass to Fourier transforms by using Definition \ref{Fou_trans_distr}. By letting 
$$(\phi_\epsilon * \mu )*( \phi_\epsilon * \tilde \mu) = \hat H_\epsilon, \,\, 
H_\epsilon \in \mathcal{S}(\mathbf{R}^n)$$ we have, by the inversion formula, that
\begin{eqnarray*}
\overline{H_\epsilon} & = &\frac{1}{(2\pi)^n} \left(\overline{(\phi_\epsilon * \mu )* (\phi_\epsilon * \tilde \mu)} \right) \hat{} \\
                     & = & \frac{1}{(2\pi)^n} \left(\phi_\epsilon * \mu * \phi_\epsilon * \tilde \mu \right) \hat{} \mbox{ (since }\phi \mbox{ is a real function)} \\
& = &  \frac{1}{(2\pi)^n} (\hat\phi_\epsilon \times \hat\mu \times \hat \phi_\epsilon \times \hat{\tilde \mu}) \mbox{ from equation (\ref{conv_prodd}) }\\
 & = & \frac{1}{(2\pi)^n} (\vert \hat\phi_\epsilon  \vert ^2 \vert \hat \mu \vert^2 ) \mbox{ since }\hat{\tilde \mu} = \overline{\hat \mu} \\
\end{eqnarray*}             
Therefore, 
\begin{eqnarray*}
\int \int k*\psi_\epsilon (x-y) d\mu(x) d\mu(y) &=& \int k(t)\int \phi_\epsilon * \mu(t - z) \phi_\epsilon * \tilde \mu(z) dz dt \\
& = & \int k(t) \hat H_\epsilon(t) dt \\
& = & \int \hat k(t) H_\epsilon(t) dt \mbox{ (by Definition \ref{Fou_trans_distr})}\\
& = & \frac{1}{(2\pi)^n} \int \hat k(t) \vert \hat\phi_\epsilon (t) \vert ^2 \vert \hat \mu (t)\vert^2 dt \\
\end{eqnarray*}
Since for $\epsilon \to 0$, $\hat \phi_\epsilon (t) = \hat f(t\epsilon )\to 1 $ (the Fourier transform of the Dirac distribution $\delta$), it follows that 
$$\liminf_{\epsilon \to 0} \int \hat k(t) \vert \hat\phi_\epsilon (t) \vert ^2 \vert \hat \mu (t)\vert^2 dt  = \frac{1}{(2\pi)^n} \int \hat k(t) \vert \hat \mu (t)\vert^2 dt.$$
Therefore 
$$I_\alpha(\mu) \leq \frac{1}{(2\pi)^n} \int \hat k(t) \vert \hat \mu (t)\vert^2 dt.$$
(4) Let us now show that 
\begin{eqnarray} 
I_\alpha(\mu) \geq \frac{1}{(2\pi)^n} \int \hat k(t) \vert \hat \mu (t)\vert^2 dt.
\end{eqnarray}
For this end it is sufficient to show that  \begin{eqnarray} \label{eqnimporttt}
\limsup_{\epsilon \to 0}\int k*\psi_\epsilon (x-y) d\mu(x) d\mu(y) \leq I_\alpha(\mu).
\end{eqnarray}
For $x-y \ne 0$, 
\begin{eqnarray*}
k*\psi_\epsilon (x-y)& =& \epsilon^{-n} \int k(x-y-h) f*f(h/\epsilon)dh \\
&=&  \int k(x-y-\epsilon h) f*f(h)dh \, (\mbox{ variable change}).\\
\end{eqnarray*}
A useful observation is that for any  $z, u \in \mathbf{R}^n$ such that $\vert z \vert = 1$ and $\vert u \vert <1/2$, one has that 
\begin{eqnarray} \label{ineq_impo}
\frac{1}{\vert z - u \vert^\alpha} \leq 1 + 2\vert u \vert.
\end{eqnarray}
Indeed, it is clear that among the $u$'s such that $\vert u \vert = \beta < 1/2$, the one which maximizes the function $1/\vert z - u \vert^\alpha$ is $u = \beta z$ (for fixed $z$). In that case,
$$\frac{1}{\vert z - u \vert^\alpha} = \frac{1}{\vert z \vert^\alpha (1 - \beta)^\alpha} = \frac{1}{(1 - \beta)^\alpha}.$$
The Taylor expansion of 
\begin{eqnarray*}
(1-\beta)^{-\alpha}  &=& 1 + \beta\left( \alpha + \frac{\alpha(\alpha +1)}{2!}\beta + \frac{\alpha(\alpha +1)(\alpha +2)}{3!}\beta^2 + \cdots\right)\\
&\leq &  1 + 2 \beta \mbox{ ( since }\alpha < 1, \,\beta < 1/2).
\end{eqnarray*}
Now we write 
$$\int \int k*\psi_\epsilon (x-y) d\mu(x) d\mu(y) = J_1 + J_2 $$
where 
\begin{eqnarray*}
J_1 &=& \int \int \int_{\{h\in B(1):\sqrt{\epsilon} \vert h \vert \leq \vert x - y \vert\}} k(x-y-\epsilon h) f*f(h)dh d\mu(x) d\mu(y); \\
J_2 &= &\int \int \int_{\{h\in B(1):\sqrt{\epsilon} \vert h \vert > \vert x - y \vert\}}k(x-y-\epsilon h) f*f(h)dh d\mu(x) d\mu(y).
\end{eqnarray*}
We compute $J_1$ by noting that if $0< \epsilon < 1/4$, $\vert h \vert < 1$ and $\sqrt{\epsilon} \vert h \vert \leq \vert x - y \vert$ then 
$$\frac{\epsilon \vert h \vert}{\vert x- y \vert} \leq \sqrt{\epsilon} \leq \frac{1}{2}.$$ Now we have that 
\begin{eqnarray*}
k(x- y - \epsilon h) & = & \frac{1}{\vert x-y \vert^\alpha} \frac{1}{\left\vert\frac{x- y}{\vert x- y \vert} - \frac{\epsilon h}{\vert x- y \vert}\right\vert^\alpha}\\
                    & \leq & \frac{1}{\vert x-y \vert^\alpha} \left(1 + \frac{2\epsilon \vert h \vert}{\vert x- y \vert}\right), \mbox{ by relation }(\ref{ineq_impo}),\\
                    & \leq & \frac{1}{\vert x-y \vert^\alpha}(1 + 2\sqrt{\epsilon})
\end{eqnarray*}
Then 
$$J_1 \leq (1 + 2\sqrt{\epsilon}) I_\alpha(\mu) \, \mbox{ since }\int f*f(h)dh  = 1.$$
To compute $J_2$ we note that for fixed $x$ and $y$ and  for $ h\in B(1)$ with $\sqrt{\epsilon} \vert h \vert > \vert x - y \vert\}$, the function  $1/\vert x-y-\epsilon h\vert^\alpha$ is maximum for $h = (x - y)/\sqrt{\epsilon}$. 
(In general for $u, v \in \mathbf{R}^n,\, \vert u - v \vert$ is minimal for $u = H v$ for some  positive constant $H$.)\\
Therefore, for $h\in B(1)$ and $\sqrt{\epsilon} \vert h \vert > \vert x - y \vert,$ 
$$ \frac{1}{\vert x-y-\epsilon h\vert^\alpha} \leq \frac{1}{(1 - \sqrt{\epsilon})^\alpha\vert x-y \vert^\alpha} =  \frac{c}{\vert x-y \vert^\alpha}.$$ 
and hence
\begin{eqnarray*}
J_2 &\leq &c \int \int \int_{\{h\in B(1):\sqrt{\epsilon} \vert h \vert > \vert x - y \vert\}}\frac{1}{\vert x-y \vert^\alpha} f*f(h)dh d\mu(x) d\mu(y)\\
&\leq & c \int_{y\in \mathbf{R}^n} \int_{h \in B(1)} \int_{\{x:\sqrt{\epsilon} > \vert x - y \vert\}}\frac{1}{\vert x-y \vert^\alpha} f*f(h) d\mu(x) dhd\mu(y) \mbox{ (Fubini's theorem)} \\
&\leq & c \int_{y\in \mathbf{R}^n} \int_{\{x:\sqrt{\epsilon} > \vert x - y \vert\}}\frac{1}{\vert x-y \vert^\alpha} d\mu(x) d\mu(y)\mbox{ (since }\int_{B(1)}f*f(h)dh = 1)
\end{eqnarray*}
It follows that
$J_1 \to I_\alpha(\mu)$ and $J_2 \to 0$ as $\epsilon \to 0$ and relation (\ref{eqnimporttt}) is proven. \qed
\end{proof}

\paragraph{Acknowlegdment} This paper is based on my PhD project and I thank my promotors W. L. Fouch\'e  and P. H. Potgieter for their guidance.


\begin{thebibliography}{99}
 \bibitem{Al-Gwaiz} Al-Gwaiz, M.A. \emph{Theory of distribution.} Marcel Dekker, 1992.
 
 \bibitem{Carleson} Carleson, L. \emph{Selected problems on exceptional sets}. Van Nostrand, 1967.
 
 \bibitem{Frostman} Frostman, O. \emph{Potentiel d'\'equilibre et capacit\'e des ensembles}. PhD thesis, Lund, 1935.
 
 \bibitem{Kahane} Kahane, J-P. \emph{Some random series of functions}. Cambridge University Press, 2nd ed., 1985.
 
 \bibitem{Mattila} Mattila, P. \emph{Geometry of sets and measures in Euclidean spaces}. Cambridge University Press, 1995.
 
 \bibitem{Rudin} Rudin, W.  \emph{Real and Complex Analysis}. McGraw-Hill, 1987.
 
 \bibitem{Strichartz} Strichartz, R. \emph{A guide to distribution theory and Fourier transform}. CRC Press, 1994.
\end{thebibliography}
\end{document}